\documentclass{amsart}
\usepackage{pstricks,pst-node}
\usepackage{amsmath,amssymb,multicol}

\usepackage{amsmath,amssymb,multicol}

\newcommand{\fieldsize}{14} 
\newcommand{\fieldsizebp}{\fieldsize bp}
\newcommand{\stonesmaller}{2bp}
\newcommand{\textstoneraise}{-4bp}
\newcommand{\polycellsize}{6} 

\newcommand{\fieldend}{%
\vrule width\fieldsizebp height0mm depth0mm\vrule width0mm height\fieldsizebp\relax%
}

\newcommand\specialps{} 
\begingroup
\catcode`\"=12 
\gdef\specialps#1{\special{" field-setup #1}}
\endgroup

\newcommand{\fieldps}[1]{%
  \vrule width0mm height\fieldsizebp\relax
  \specialps{#1}%
}

\newcommand{\fieldcenter}[1]{\rlap{\vbox to\fieldsizebp{\vfil\hbox to\fieldsizebp{\hfil \strut #1\hfil}\vfil}}}
\newcommand{\fieldcenterwhite}[1]{
\fieldcenter{#1}
}

\newcommand{\textstone}[1]{\raisebox{\textstoneraise}{%
\kern-\stonesmaller#1\vrule width\fieldsizebp height0mm depth0mm\relax\kern-\stonesmaller%
}}

\newcommand{\extN}{}
\newcommand{\extS}{}
\newcommand{\extE}{}
\newcommand{\extW}{}
\newcommand{\fif}{ff \extN\extS\extE\extW}

\newcommand{\nostone}{\fieldps{\fif}\fieldend}

\newcommand{\nostoned}{\fieldps{\fif do}\fieldend}

\newcommand{\fieldstonew}[1]{\fieldps{\fif sw}\fieldcenter{\scriptsize #1}\fieldend}
\newcommand{\fieldstoneb}[1]{\fieldps{\fif sb}\fieldcenterwhite{\scriptsize #1}\fieldend}

\newcommand{\nothing}{\fieldend}

\newcommand{\letter}[1]{\fieldps{\fif}\fieldcenter{{\footnotesize #1}}\fieldend}

\newcommand{\ministonew}{\fieldps{\fif mw}\fieldend}
\newcommand{\ministoneb}{\fieldps{\fif mb}\fieldend}
\newcommand{\ministoned}{\fieldps{\fif md}\fieldend}

\newcommand{\paveup}{\fieldps{\fif pu}\fieldend}
\newcommand{\pavedown}{\fieldps{\fif pd}\fieldend}
\newcommand{\paveleft}{\fieldps{\fif pl}\fieldend}
\newcommand{\paveright}{\fieldps{\fif pr}\fieldend}
\newcommand{\paveupdown}{\fieldps{\fif pud}\fieldend}
\newcommand{\paveleftright}{\fieldps{\fif plr}\fieldend}

\newcommand{\fieldcoord}[1]{\vrule width0mm height\fieldsizebp\relax\fieldcenter{#1}\fieldend}

\makeatletter

\newcommand{\boardcatcodes}{%
\catcode`\^^M\active
\catcode`\ \active
\catcode`\.\active
\catcode`\O\active
\catcode`\X\active
\catcode`\0\active
\catcode`\1\active
\catcode`\2\active
\catcode`\3\active
\catcode`\4\active
\catcode`\5\active
\catcode`\6\active
\catcode`\7\active
\catcode`\8\active
\catcode`\9\active
\catcode`\*\active
\catcode`\+\active
\catcode`\o\active
\catcode`\x\active
\catcode`\q\active
\catcode`\a\active
\catcode`\b\active
\catcode`\c\active
\catcode`\d\active
\catcode`\e\active
\catcode`\f\active
\catcode`\g\active
\catcode`\h\active
\catcode`\>\active
\catcode`\<\active
\catcode`\^\active
\catcode`\v\active
\catcode`\|\active
\catcode`\-\active
\catcode`\\\active
\catcode`\N\active
\catcode`\S\active
\catcode`\E\active
\catcode`\W\active
}

\def\definitionstack{}

\long\def\addboarddefinition#1#2{%
\expandafter\def\expandafter\definitionstack\expandafter{\definitionstack%
\def#1{#2}%
}%
}

\def\boarddefinition{%
\begingroup%
\boardcatcodes%
\boarddefinition@%
}
\def\boarddefinition@#1{%
\endgroup%
\addboarddefinition{#1}
}

\boarddefinition{N}{\def\extN{N }}
\boarddefinition{S}{\def\extS{S }}
\boarddefinition{E}{\def\extE{E }}
\boarddefinition{W}{\def\extW{W }}
\boarddefinition{ }{\doline\doboardspace}
\boarddefinition{+}{\doline\nostone\eatspacemode}
\boarddefinition{.}{\doline\nostoned\eatspacemode}
\boarddefinition{*}{\doline\nostoned\eatspacemode}
\boarddefinition{O}{\doline\stonemodes{\fieldstonew}}
\boarddefinition{X}{\doline\stonemodes{\fieldstoneb}}
\boarddefinition{o}{\doline\ministonew\eatspacemode}
\boarddefinition{x}{\doline\ministoneb\eatspacemode}
\boarddefinition{q}{\doline\ministoned\eatspacemode}
\boarddefinition{
}{\donewline} 
\boarddefinition{0}{\doline\savenumbereven\doboardnumber{0}}
\boarddefinition{1}{\doline\savenumberodd\doboardnumber{1}}
\boarddefinition{2}{\doline\savenumberodd\doboardnumber{2}}
\boarddefinition{3}{\doline\savenumberodd\doboardnumber{3}}
\boarddefinition{4}{\doline\savenumbereven\doboardnumber{4}}
\boarddefinition{5}{\doline\savenumberodd\doboardnumber{5}}
\boarddefinition{6}{\doline\savenumbereven\doboardnumber{6}}
\boarddefinition{7}{\doline\savenumberodd\doboardnumber{7}}
\boarddefinition{8}{\doline\savenumbereven\doboardnumber{8}}
\boarddefinition{9}{\doline\savenumberodd\doboardnumber{9}}
\boarddefinition{a}{\doline\doboardletter{a}}
\boarddefinition{b}{\doline\doboardletter{b}}
\boarddefinition{c}{\doline\doboardletter{c}}
\boarddefinition{d}{\doline\doboardletter{d}}
\boarddefinition{e}{\doline\doboardletter{e}}
\boarddefinition{f}{\doline\doboardletter{f}}
\boarddefinition{g}{\doline\doboardletter{g}}
\boarddefinition{h}{\doline\doboardletter{h}}
\boarddefinition{>}{\doline\paveright\eatspacemode}
\boarddefinition{<}{\doline\paveleft\eatspacemode}
\boarddefinition{^}{\doline\paveup\eatspacemode}
\boarddefinition{v}{\doline\pavedown\eatspacemode}
\boarddefinition{|}{\doline\paveupdown\eatspacemode}
\boarddefinition{-}{\doline\paveleftright\eatspacemode}
\boarddefinition{\}{\dobackslash}
{{{{{{}}}}}}  

\newcommand{\activateboard}{%
\begingroup
\definitionstack
\boardcatcodes
}

\newcommand{\deactivateboard}{%
\coordx\endgroup%
}

\def\activateboarddefs{%
\def\normalnewlinemode{\def\donewline{\doline\doboardspace\par\normalspacemode\beginlinemode}}
\def\eatnewlinemode{\def\donewline{\normalnewlinemode}}
\eatnewlinemode
\def\normalspacemode{\def\doboardspace{\nothing\eatspacemode}}%
\def\specialspacemode##1{%
  \def\doboardspace{##1\normalmodes}%
}%
\def\eatspacemode{\specialspacemode{}}%
\def\savenumberodd{\def\reallydonumber####1{\fieldstoneb{####1}}}%
\def\savenumbereven{\def\reallydonumber####1{\fieldstonew{####1}}}%
\def\normalnumbermode{\def\doboardnumber####1{%
   \def\savenumber{####1}%
   \specialspacemode{\reallydonumber{\savenumber}}%
   \specialnumbermode{\reallydonumber{\savenumber\extranumber}}%
}}%
\def\specialnumbermode##1{%
  \def\doboardnumber####1{\def\extranumber{####1}##1\normalmodes}%
}
\def\normallettermode{\def\doboardletter####1{%
  \def\saveletter{####1}%
  \specialspacemode{\letter{\saveletter}}%
  \specialnumbermode{\letter{\saveletter$_{\extranumber}$}}%
}}%
\def\speciallettermode##1{%
  \def\doboardletter####1{\def\extraletter{####1}##1\normalmodes}%
}
\def\stonemodes##1{%
  \specialspacemode{##1{}}%
  \speciallettermode{##1{\extraletter}}%
}%
\def\normalmodes{\normalspacemode\normalnumbermode\normallettermode}
\normalmodes
\def\dobackslash{\deactivateboard\@ifnextchar e{\readend}{\readnotes}}
\def\readend end{\end}
\def\readnotes notes{%
\endgroup 
\begingroup 
\parindent0mm
\small}
\def\beginlinemode{\def\doline{\coordy\normallinemode}}%
\def\normallinemode{\def\doline{}}%
\beginlinemode%
\def\nocoordinates{%
  \def\coordx{}%
  \def\coordy{}%
  \activateboard%
}%
\def\readcoordinates(##1,##2){%
  \def\coordx{%
    \nothing%
    \setcounter{coordcnt}{1}%
    \loop\fieldcoord{\alph{coordcnt}}\ifnum\c@coordcnt<##1\addtocounter{coordcnt}{1}\repeat%
    \par
  }%
  \def\coordy{%
    \fieldcoord{$\arabic{coordcnt}$}%
    \addtocounter{coordcnt}{-1}%
  }%
  \setcounter{coordcnt}{##2}%
  \activateboard%
}%
\def\readmorecoordinates(##1,##2)(##3,##4){%
  \def\coordx{%
    \nothing%
    \setcounter{coordcnt}{##1}%
    \loop\fieldcoord{{\small$\arabic{coordcnt}$}}\ifnum\c@coordcnt<##3\addtocounter{coordcnt}{1}\repeat%
    \par
  }%
  \def\coordy{%
    \fieldcoord{{\small$\arabic{coordcnt}$}}%
    \addtocounter{coordcnt}{-1}%
  }%
  \setcounter{coordcnt}{##4}%
  \activateboard%
}%
}

\makeatother

\newcommand{\boardwidth}{0.43\textwidth}

\newenvironment{bboard}{
\begin{minipage}{\boardwidth}
\vskip1ex%
\begingroup%
\parindent0mm%
\baselineskip0mm%
\lineskiplimit0mm%
\lineskip0mm%
\activateboarddefs%
}{%
\endgroup%
\vspace*{1ex}%
\end{minipage}\hspace{\fill}%
}

\newcounter{coordcnt}

\newlength{\polycelllength}
\setlength{\polycelllength}{\polycellsize bp}

\newcommand{\polyo}[2]{\strut\specialps{#1}\kern#2\polycelllength}

\newtheorem{Theo}{Theorem}
\newtheorem{Lem}[Theo]{Lemma}
\newtheorem{Cor}[Theo]{Corollary}
\newtheorem{Con}[Theo]{Conjecture}
\newtheorem{prop}[Theo]{Proposition}

\newcommand{\co}[1]{c_{\mathrm{#1}}}

\newcommand{\bobox}[1]{\vrule\hbox to 1.7ex{\hfil\ensuremath{\scriptstyle{#1}}\hfil}}
\newcommand{\boline}[3]{\hbox{\bobox{#1}\bobox{#2}\bobox{#3}\vrule height 1.4ex depth .3ex}}
\newcommand{\boix}[9]{%
\raise-2.5ex\vbox{%
\vskip0.3ex
\hrule
\boline{#7}{#8}{#9}
\hrule
\boline{#4}{#5}{#6}
\hrule
\boline{#1}{#2}{#3}
\hrule
\vskip0.3ex
}}
\newcommand{\bovi}[6]{%
\raise-1.5ex\vbox{%
\vskip0.3ex
\hbox{\vrule\hskip 1.7ex\vrule\hskip 1.7ex\vrule\hskip 1.7ex\vrule height .1ex depth .3ex}
\hrule
\boline{#4}{#5}{#6}
\hrule
\boline{#1}{#2}{#3}
\hrule
\vskip0.3ex
}}

\newcommand{\bnode}[1]{\rnode{#1}{\bullet}}

\newcommand{\Abar}{\overline{A}}
\newcommand{\calN}{\mathcal{N}}
\newcommand{\calT}{\mathcal{T}}

\newcommand{\N}{\mathbb{N}}
\newcommand{\Z}{\mathbb{Z}}
\newcommand{\Q}{\mathbb{Q}}

\newcommand{\surject}{\twoheadrightarrow}

\begin{document}

\author[G. Bhowmik]{Gautami Bhowmik}
\author[I. Halupczok]{Immanuel Halupczok}
\thanks{The second author was supported by the Agence National de la Recherche
(contract ANR-06-BLAN-0183-01).}
\author[J.-C. Schlage-Puchta]{Jan-Christoph Schlage-Puchta}
\title[Zero-sum free sequences]{Inductive Methods and Zero-sum free sequences}
\begin{abstract} 
A fairly long standing conjecture was that the Davenport constant of a group
 $G=\Z_{n_1}\oplus\dots\oplus\Z_{n_k}$ with
$n_1|\dots|n_k$ is $1+\sum_{i=1}^k (n_i-1)$. This conjecture is false
in general, but the question remains for which groups it is true.
By using inductive methods we prove that for two fixed integers $k$
and $\ell$ it is possible to decide whether the conjecture is
satisfied for all groups of the form $\Z_k^{\ell}\oplus \Z_{n}$ with $n$ co-prime
to $k$.

We also prove the conjecture
for groups of the form $\Z_3\oplus\Z_{3n}\oplus\Z_{3n},$ where $n$ is co-prime
to $6$, assuming a conjecture about the maximal zero-sum free sets in $\Z_n^2$.
\end{abstract}

\maketitle

MSC-Index 11B50, 20K01, 20F10, 20D60

\section{Introduction and Results}

Let $G$ be a finite abelian group written additively, $a_1, \ldots,
a_k$ a sequence of elements 
in $G$. This sequence contains a zero-sum if there is some
non-empty subsequence $1\leq i_1<i_2<\dots<i_\ell\leq k$ satisfying
$a_{i_1}+\dots+a_{i_\ell}=0$, otherwise it is called zero-sum
free. Denote by $D(G)$ the least integer $k$ such that every sequence
of length $k$ contains a zero-sum, this number is usually called
Davenport's constant, since the question of whether zero-sums exist
was  studied by Davenport in the context of algebraic number
theory (where $G$ is the class group of some number field, the
elements $a_i$ are given ideal classes from which one wants to
construct a principal ideal). This line of research was continued in
the study of domains with non-unique factorisation, for an overview
see \cite{GHK}. Among applications, Br\"udern and Godinho \cite{Br} discovered that
the existence of zero-sums can be used to simplify $p$-adic forms,
which led to considerable progress towards Artin's conjecture on
$p$-adic forms. 

To avoid cumbersome notation we shall from now on always talk about
multi-sets instead of sequences; in the sequel all sets are multi-sets
unless explicitly stated otherwise. We shall write the multiplicity of
an element as its exponent, e.g. $\{a^n, b^m\}$ is a multi-set containing
$n+m$ elements, $n$ of which are equal to $a$, and $m$ are equal to
$b$. We believe that the imprecision implied by the non-standard use
of equality is more than outweighed by easier readability.

One approach to bound $D(G)$ is the so called inductive method, which
runs as follows: If $N<G$ is a subgroup and
$n$ an integer such that every sequence of length $n$ in $G/N$
contains a system of $D(N)$ disjoint zero-sums, then $D(G)\leq n$. In
fact, each zero-sum in $G/N$ defines an element in $N$, and choosing a
zero-sum among these elements defines a zero-sum in
$G$. Unfortunately, in general this method does not give the exact
value for $D(G)$. For example, for $G=\Z_3^2\oplus\Z_{3n}$, Delorme,
Ordaz and Quiroz showed that $D(G)\leq 3n+5$, which is 1 more than the
exact value.
The sub-optimality of this method stems from the fact that in general
we have many ways to choose a system of disjoint zero-sums, and it
suffices to show that one of these systems yields a zero-sum in
$N$. If the structure of all zero-sum free subsets in $N$ of size close to
$D(N)$ is sufficiently well understood one can use this information
to choose an appropriate system of subsets in $G/N$. In this way one
can show that for groups of the form $G=\Z_3^2\oplus\Z_{3n}$ we always
have $D(G)=3n+4$ (confer \cite{Montreal}); the corresponding lower
bound being given by the mulitset $\{(1, 0, 0)^2, (0, 1, 0)^2, (0, 0,
1)^{3n-1}\}$. In fact, this example immediately generalises to
arbitrary finite groups: If $G=\Z_{n_1}\oplus\dots\oplus\Z_{n_k}$ with
$n_1|\dots|n_k$, then $D(G)\geq M(G) := 1+\sum_{i=1}^k (n_i-1)$. The
conjecture that $D(G)=M(G)$, which we shall refer to as the main
conjecture, is proven for groups of rank 2, and fails for infinitely 
many groups of rank $\geq 4$. It is not yet known whether it holds
true for all groups of rank 3. 

In this article we generalise the improved inductive method to other
sequences of groups. We first give a decidability result.
Suppose $k, \ell \in \N$ are fixed. Then one can check the
main conjecture for all groups of the form $G := \Z_k^{\ell}\oplus\Z_{n}$ at
once (in a finite amount of time), where $n$ runs through all numbers
co-prime to $k$. Note that
$G\cong\Z_k^{\ell-1}\oplus\Z_{kn}$, so $M(G) = (\ell-1)\cdot(k-1) + kn$.
Moreover, we give a description of the set of numbers $n$
such that the main conjecture fails for $\Z_k^{\ell}\oplus\Z_{n}$.

It turns out that the same proof even yields a bit more: if
the main conjecture turns out to be false for $G$, then one can ask
about the difference $D(G)-M(G)$. Our results not only
apply to the set of $n$ where the main conjecture fails, but also
to set set of $n$ where $D(G)-M(G) > \delta$ for any fixed $\delta$.
Here is the precise statement:

\begin{Theo}
\label{thm:decidable}
Suppose $k \ge 2$, $\ell \ge 1$ and $\delta$ are three integers. Let $\calN$ be the set
of integers $n$ co-prime
 to $k$ such  that
$D(\Z_k^{\ell}\oplus\Z_{n})>kn + \delta$.
Then either $\calN$ is finite, or
there exists an integer $d>0$ and a set
$\calT$ of divisors of $d$ containing 1 such that
$\calN$ differs from the set
\[
\calN' := \{x\in \N \mid (x, d) \in \calT\}
\]
only in finitely many elements.

In addition, there is an algorithm which, given $k$, $\ell$ and $\delta$,
prints out $\calN$ if the latter is finite. Otherwise its output is
$d$, $\calT$ and the set of elements in which $\calN$ and $\calN'$
differ.
\end{Theo}

Choosing $\delta = (\ell-1)\cdot(k-1)$ yields:

\begin{Cor}
\label{cor:decidable}
Suppose $k \ge 2, \ell \ge 1$ are two integers. Let $\calN$ be the set
of integers $n$ co-prime to $k$ such that the main conjecture fails for
$\Z_k^{\ell}\oplus\Z_{n}$.
Then $\calN$ has the form described in Theorem~\ref{thm:decidable}, and there
is an algorithm which, given $k$ and $\ell$, describes $\calN$ as above.
\end{Cor}

In theory, this means that a computer can be programmed to
prove statements of the form ``the main conjecture is true for
$\Z_k^{\ell}\oplus\Z_{n}$ for all $n$ co-prime to $k$''.
However, the reader should be aware that the existence of an algorithm
often sounds better than it is: a straight-forward application of our
algorithm would require
astronomical running time even for very small $k$ and $\ell$ (see
constants appearing in Proposition~\ref{prop:bound}). Still, we
believe that by combining computer search with manual arguments one
can prove the main conjecture for certain series of groups.
In fact, in \cite{Montreal} the methods of this theorem have been explicitly
applied to prove the main conjecture in the case $k = 3$, $\ell = 3$.

In the theorem, we mention that the set $\calT$ of divisors contains
$1$. This is helpful to get a statement of the form ``if there is a
counter-example to the main conjecture, then there is a small one'';
indeed, Proposition~\ref{prop:bound} is such a statement.

The proof of Theorem~\ref{thm:decidable} makes much use of the simple structure of
$\Z_n$ where there is essentially one single example of a large zero-sum
free set. In our next theorem, we would like to replace $\Z_n$ by a
larger group.
However, for non-cyclic groups the structure of maximal zero-sum free
sets is less clear and there are essentially different possibilities
for such sets. Due to this complication, we can only deal with groups of rank 2. 
Though the structure
of maximal zero-sum free sets is not known, there is a
plausible conjecture concerning these sets. We say that an integer $n$
satisfies property $B$ if every zero-sum free subset
$A\subseteq\Z_n^2$ contains an element $a$ with multiplicity $\geq
n-2$. 

\begin{Con}
\label{Con:B}
Every integer $n$ satisfies property $B$.
\end{Con}

This conjecture is known to hold in several cases.

\begin{prop}
\label{prop:B}
\begin{enumerate}
\item If $n$ and $m$ satisfy property $B$, then so does $nm$.
\item All prime numbers up to $23$ satisfy property $B$.
\end{enumerate}
\end{prop}

The first statement is essentially due to Gao, Geroldinger and Schmid
\cite{GGS}, the second is proven in \cite{Bmult}.

\begin{Theo}
\label{thm:Davenport}
Let $n$ be an integer co-prime to $6$ such that $B(n)$ holds true.
Then $D(\Z_3\oplus\Z_{3n}^2)=6n+1$.
\end{Theo}

We remark that even the simplest case dealt by this theorem, that
is $\Z_3\oplus\Z_{15}^2$, was till now undecided.

Although we tried to prove as much as possible by hand,
the proof of this theorem needs a lemma on subsets of $\Z_{3}^{3}$
which we could only prove by massive case distinction, which has been
done by our computer.

\section{Auxiliary results}

For an abelian group $G$, we denote by $D_m(G)$ the minimal $n$ such that
any subset of $G$ of cardinality $n$ contains $m$ disjoint zero-sums.

\begin{Lem}
\label{Lem:Dm}
\begin{enumerate}
\item
Let $k$ and $\ell$ be integers. Then there exists a constant $c(k,\ell)$
such that $D_m(\Z_k^\ell)\le km+c(k,\ell)$.
\item
We have $D_m(\Z_3^2)=3m+2.$
\end{enumerate}
\end{Lem}

\begin{proof}
(1)
Form as many zero-sums as possible which are of the form $\{a^k\}$ for some
$a \in \Z_k^\ell$. For
each $a \in \Z_k^\ell$, there are at most $k-1$ copies of $a$ in $\Abar$
which we can not use in this way, so
$c(k,\ell) := (k-1)\cdot k^\ell$ is certainly sufficient.

(2)
It is easy to check that every subset of 5 elements contains a
zero-sum, and that every subset of 7 elements contains a zero-sum of
length $\leq 3$. Our claim now follows by induction on $m$.
\end{proof}

\begin{Lem}
\label{Lem:LinearSystems}
Let $k, \ell$ be integers, $A\in\Z^{k\times \ell}$ a matrix, $b\in\Z^k$ a
vector. Then 
either (a) there exists an integer $d$, and a set 
$\mathcal{T}$ of divisors of $d$ including 1,
such that the system $Ax=b$ is solvable in $\Z_n$ if and only if $(d,
n)\in\mathcal{T}$ or (b) there exists a finite set of integers 
$\mathcal{N}$, such that the above system is solvable if and only if 
$n\in\mathcal{N}$.

If all entries in $A$ are of modulus $\leq M$, and all entries of $b$
are of modulus $ \le N$, then in case (a) $d\leq \min(k, \ell)! M^{\min(k, \ell)}$,
and there is a polynomial $p$,
independent of $k$, $\ell$, $N$ and $M$, such
that  in case (b), every element $x\in\mathcal{N}$ satisfies $x\leq 
N2^{p(k\ell\log M)}$.

\end{Lem}
\begin{proof}
Computing the Smith normal form of the matrix $A$, we see that there
exist invertible matrices $P, Q$ over $\Z$, such that $D=PAQ^{-1}$ has
non-zero entries at most on the diagonal $d_{ii}$ , $i\leq k$, and
these entries satisfy $d_{ii}|d_{i+1,i+1}$. Since every matrix
invertible over $\Z$ is also invertible over $\Z_n$, the equation $Ax=b$
is solvable in $\Z_n$ if and only if the equation $Dx=b'$ is solvable,
where $b'=Pb$. A necessary condition for solvability is that
in every row containing only zeros in $D$, the corresponding entry of
$b'$ vanishes, that is, $n|b'_j$ for every $j$ such that $j>m$, where
$m$ is the greatest integer such that $d_{mm}\neq 0$. If
one of these $b'_j$ does not vanish, then there are at 
most finitely many $n$ for which the equation is solvable, and our
claim is true. If all these $b'_j$ equal zero, the system is
equivalent to the system $d_{ii}x_i = b'_i$, which is solvable if and
only if $(n, d_{ii}) | b'_i$. We take $d$ to be $d_{mm}$. Since
$d_{ii}|d$ for each $i\leq m$, the set of $n$ for which the system is
solvable is of the form $\{n \mid (n,d) \in \mathcal{T}\}$ for some set
$\mathcal{T}$. Moreover  $(n,d) = 1$ implies
$(n, d_{ii}) | b'_i$, so $1 \in \mathcal{T}$.

For the numerical bounds note that $d$ equals the greatest common
divisor of all $m\times m$ sub-determinants of $A$. Since the
$\Q$-rank of $A$ equals $m$, there exists a non-vanishing
sub-determinant, containing only entries $\leq M$, which is therefore
$\leq m! M^m \leq \min(k, \ell)! M^{\min(k, \ell)}$.

The entries in the set $\mathcal{N}$ are bounded by the entries
in $Pb$, which in turn are bounded by $kN$ times the entries of $P$.
A general estimate for the entries of such transformation matrices was obtained
by Kannan and Bachem \cite[Theorem~5]{KB}. They found a polynomial algorithm
which takes an $\ell'\times \ell'$-matrix $A$ with
integral entries, transforms it into Smith normal form $PAQ^{-1}$, and
returns the transformation matrices $P$ and $Q$. To apply this to our case,
enlarge our $A$ to a square matrix by adding zeros (i.e. $\ell' = \max(k,\ell)$). Then
the size of the input data is $(\ell')^2\log M$, so the size of the output
data---and in particular the number of digits of the entries of $P$ and
$Q$---is bounded by $p(\ell'\log M)$ for some polynomial $p$. After possibly changing
$p$, this yields the claim.
\end{proof}
\begin{Cor}
\label{Cor:LinearSystem}
Consider the system $Ax=b$ as in the previous lemma,
set $m := \min(k,\ell)$, and suppose that
there are infinitely many $n$ such that this system is solvable in
$\Z_n$. Then for each $z\ge z_0 = \max\big(21, \frac{m \log (mM)}{\log
  2}\big)$ the system is solvable for some $n\in[z, 2z]$.
\end{Cor}
\begin{proof}
If the system has infinitely many solutions, then there exists an
integer $d\leq m! M^m$ such that the system is solvable in $\Z_n$
whenever $(n, d)=1$. If the system is unsolvable for all $n\in[z,
2z]$, then in particular, $d$ is divisible by all prime
numbers in this interval. Since for $z\geq 21$, the product of all
prime numbers in $[z, 2z]$ is $\geq 2^z$, our claim follows.
\end{proof}
The following result is essentially due to Bovey, Erd\H os and
Niven \cite{BEN}.
\begin{Lem}
\label{Lem:CD}
Let $A\subseteq\Z_n$ be a zero-sum free multi-set containing $N$
elements, where $N\ge 2n/3$. Then there exists an element $a$ of $\Z_n$,
which occurs in $A$ with multiplicity greater than $2N-n$. Moreover,
$a$ is a generator of $\Z_n$. 
\end{Lem}
\begin{proof}
The statement on the multiplicity is \cite{BEN}. Now suppose that $a$
is not a generator of $\Z_n$, and let $H$ be the subgroup generated
by $a$. Denote by $m$ the multiplicity of $a$. Among $(\Z_n:H)$
elements of $\Z_n/H$ we can choose a zero-sum, that is, among the $N-m$
elements of $A \setminus \{a^m\}$ we can choose a system of
$\lfloor\frac{N-m}{(\Z_n:H)}\rfloor$ disjoint sets, each one adding up
to an element in $H$. Since $A$ is zero-sum free, we cannot obtain
$|H|$ elements in this way, that is,
$m+\lfloor\frac{N-m}{(\Z_n:H)}\rfloor\leq |H|-1$, which implies
$(\Z_n:H)m+N-m<n$. Since $m\geq 2N-n+1$, and $(\Z_n:H)\geq 2$, we
obtain $3N+1<2n$, contradicting $N \ge 2n/3$.
\end{proof}
\begin{Cor}
\label{Cor:CD}
Let $A\subseteq\Z_n$ be a subset with $|A|\geq 3n/4$. Then $A$ is
zero-sum free if and only if $0 \notin A$ and
there exists some invertible $\alpha\in \Z_n^{\times}$,
such that $\sum_{a\in A} \iota(\alpha\cdot a)\leq n-1$, where
$\iota\colon \Z_n \to \N$ is the map sending $x$ to
the least non-negative residue contained in the class $x$.
\end{Cor}
\begin{proof}
Obviously, if $0 \notin A$ and
$\sum_{a\in A} \iota(\alpha\cdot a)\leq n-1$, then $A$ is
zero-sum free. Hence, we assume that $A$ is zero-sum free and bound
the sum. 
In view of Lemma~\ref{Lem:CD} we may assume without loss that
$A$ contains the element 1 with multiplicity $m > n/2$. If $A$
contains an element in the interval $[n/2, n]$, this element can be
combined with a certain multiple of 1 to get a zero-sum. Let $x_1,
\ldots, x_k$ be the list of all elements in $A$ different from
1. Either $\sum \iota(x_i)\leq n-m-1$, which is consistent with our claim, or
there is a least $\ell$ such that $s=\sum_{i=1}^\ell \iota(x_i)>
n-m-1$. Since no single $x_i$ satisfies $\iota(x_i)> n/2$, we have
$s\in[n-m, n-1]$, hence, $s$ can be combined with a certain multiple
of 1 to get a zero-sum, which is a contradiction.
\end{proof}
\section{Proof of Theorem~\ref{thm:decidable}}

\begin{proof}[Proof of Theorem~\ref{thm:decidable}]
Let $k$ and $\ell$ be fixed once and for all. We want to describe
the set of $n$ co-prime to $k$
such that $D(\Z_k^{\ell}\oplus\Z_{n})>kn + \delta$ holds.
This is equivalent to the existence of a zero-sum free
set $A\subset \Z_k^{\ell}\oplus\Z_{n}$ of cardinality
$kn + \delta$.

First note that such a set $A$ can be described by its projection
$\Abar$ onto $\Z_{k}^{\ell}$ and the multi-function
$f\colon \Abar \to \Z_{n}$ such that $(a, f(a))\in A$ is the preimage of
$a \in \Abar$. Using this description, the existence of a set $A$ as above
is equivalent to the existence of a set
$\Abar \subset \Z_{k}^{\ell}$ of cardinality $kn + \delta$
and a multi-function $f\colon \Abar \to \Z_{n}$ (call $(\Abar, f)$ a
``candidate'') such that the following
condition holds:

(*) For any zero-sum $Z \subset \Abar$, the sum $\sum_{a\in Z}f(a)$
is not equal to zero.

The sum $\sum_{a\in Z}f(a)$ will often simply be called the
``$\Z_{n}$-sum of $Z$''.
Moreover,
we will use the following terminology: A ``constant'' is a value which
only depends on $k$, $\ell$ and $\delta$ (but not on $n$); ``bounded'' means bounded
by a constant (in the sense just described), and ``almost all'' means that
the number of exceptions is bounded.

Here is the main part of the proof. We initially skip the proofs
of the three following steps:

\begin{enumerate}
\item
Suppose $(\Abar, f)$ is a candidate and
$(Z_i)_{i\le m}$ is a system of $m$ disjoint zero-sum subsets of $\Abar$.
From this we can form the multi-set $B:= B((Z_i)_i) := \{\sum_{a\in Z_i}f(a)\mid
1 \le i \le m\} \subset \Z_n$. If $(\Abar, f)$ satisfies (*), then
$B$ has to be zero-sum free.

We will find a constant $\co{defect}$ such that $(\Abar, f)$ satisfies (*) if and only
if for all systems $(Z_i)_{i\le m}$ of $m := n - \co{defect}$ disjoint zero-sum subsets of $\Abar$,
the corresponding set $B((Z_i)_i)$ is zero-sum free.
From now on, we fix $m$ like this.
\item
We will find a constant $\co{card}$ such that 
if (*) holds for the candidate $(\Abar, f)$ and
$(Z_i)_{i}$ is a system of $m$ disjoint zero-sums of $\Abar$,
then at most $\co{card}$ of the sets $Z_i$ do not have cardinality $k$.

\item
We will show that if a candidate $(\Abar', f')$ satisfying (*) exists,
then there does already exist a candidate $(\Abar, f)$ of a particular
form. Candidates of this form will be called
``main candidates'', and they are defined as follows.
We will fix a suitable constant
$\co{var}$. $(\Abar, f)$ is a main candidate
if there exists an element $a_0 \in \Z_k^l$ such that there are at least
$|\Abar| - \co{var}$ occurrences of $a_0$ in $\Abar$ with $f(a_0) = \frac{1}{k}$.
Note that $\frac{1}{k}$ does make sense as $k$ and $n$ are co-prime.
(Right now, we could as well have written $f(a_0) = 1$ instead of
$f(a_0) = \frac{1}{k}$, but later, $\frac{1}{k}$ will be more handy.)
\end{enumerate}

The remainder of the proof goes as follows:

\begin{enumerate}\setcounter{enumi}{3}
\item
A ``datum for a main candidate'' is a tuple
$(a_0, (a_j)_j, (f_j)_j)$, where $a_0\in \Z_k^\ell$,
$(a_j)_j \in (\Z_k^\ell)^{\co{var}}$,
and $(f_j)_j \in (\Z_n^\ell)^{\co{var}}$.
Such a datum yields a main candidate $(\Abar, f)$ in the following way:
$\Abar = \Abar_0 \cup \Abar_{\star}$,
where $\Abar_0 := \{a_0^{kn + \delta - \co{var}}\}$ and
$\Abar_{\star} := \{a_j \mid 1 \le j \le \co{var}\}$,
$f(a_0)=\frac{1}{k}$
for each $a_0 \in \Abar_0$, and $f(a_j)  = f_j$ for
$a_j \in \Abar_{\star}$.
Each main candidate can be described by such a datum.
\end{enumerate}

Only the $(f_{j})_{j}$ part of such a datum depends on $n$.
Our goal now is to verify that after fixing
$a_0$ and $(a_j)_j$, whether (*) holds for the corresponding candidate
depends on $(f_{j})_{j}$ in a simple way:
we will construct systems of linear equations over $\Z$ such
that (*) holds if and only if the tuple $(f_{j})_{j}$ is a solution
of one of these systems modulo $n$.
Then the theorem will follow using Lemma~\ref{Lem:LinearSystems}.

\begin{enumerate}\setcounter{enumi}{4}
\item
Fix a datum $(a_0, (a_j)_j, (f_j)_j)$ and the corresponding
main candidate $(\Abar, f)$ as in step (4), and
suppose that $(Z_i)_i$ is a system of
$m$ disjoint zero-sum subsets of $\Abar$.
At least $m - \co{card} - \co{var} =: m - \co{eq}$ of the sets $Z_i$
are subsets of $\Abar_0$ of the form $\{a_0^k\}$.
The other $\co{eq}$ sets form a system of disjoint
zero-sums of $\Abar \setminus \{a_{0}^{k(m-\co{eq})}\}
= \Abar_{\star} \cup \{a_{0}^{k\co{defect}+\delta+(k-1)\co{eq}}\} =:
\Abar_{\star\star}$.
On the other hand, any system $(Z_{i})_{i \le \co{eq}}$ of $\co{eq}$
disjoint zero-sums of $\Abar_{\star\star}$ yields a system of $m$
disjoint zero-sums of
$\Abar$ by adding sets $Z_{i} = \{a_0^k\}$ for $\co{eq} < i \le m$.
From now on, we suppose that $(Z_i)_{i \le m}$ was obtained in this way.
\item
The set $B := B((Z_i)_{i\le m}) \subset \Z_n$ corresponding to such a system is of the form
$\{b_1, \dots, b_{\co{eq}}, 1^{m-\co{eq}}\}$,
where $b_i = \sum_{a\in Z_i}f(a) = \sum_{\{j\mid a_j\in Z_i\}}f_j
+ \frac{1}{k}z_{i}$
and $z_{i} = |Z_{i} \cap (\Abar_{\star\star} \setminus \Abar_{\star})|$.
\item
Suppose $m \ge \frac{3}{4}n$, i.e.\ $n \ge 4\co{defect}$.
Then we can apply Corollary~\ref{Cor:CD} to the set $B$ and get that it
is zero-sum free if and only if $b_i \ne 0$ for all $i \le \co{eq}$ and
there exists some $\alpha\in \Z_n^{\times}$ such that
$\sum_{b\in B} \iota(\alpha\cdot b)\leq n-1$
(with $\iota\colon \Z_n \to \N$ defined as in Corollary~\ref{Cor:CD}).
Supposing $m - \co{eq} \ge n/2$, we get that
only $\alpha = 1$ is possible, and the condition becomes
$\sum_{i=1}^{\co{eq}}\iota(b_i) < n - (m-\co{eq}) = \co{defect} + \co{eq}$.
\item
This can be reformulated as follows:
Set $C_{0} :=
\{(c_{i})_{i\le\co{eq}} \in \Z^{\co{eq}} \mid c_{i} \ge 1 \text{ and }
\sum_{i=1}^{\co{eq}}c_{i} < \co{defect} + \co{eq}\}$
(note that $C_{0}$ does not depend on $n$),
and denote by
$\pi\colon \Z^{\co{eq}} \surject \Z_n^{\co{eq}}$ the projection.
Then $B$ is zero-sum free if and only if $(b_{i})_{i} = \pi((c_{i})_{i})$
for some $(c_{i})_{i} \in C_{0}$.
Moreover, we rewrite $b_{i} = \pi(c_{i})$ as
$\sum_{\{j\mid a_j\in Z_i\}}kf_j = \pi(kc_{i} - z_{i})$.
\item
Putting all this together, we have:
For sufficiently large $n$,
there exists a pair $(\Abar, f)$ satisfying $(*)$ if and only if:
\[
\overbrace{%
  \!\!\bigvee_{\substack{a_0 \in \Z_k^\ell\\ (a_j)_j \in (\Z_k^\ell)^{\co{var}}}}
  \!\!\!\!\!\!\!\!\!\!\!\exists (f_j)_j \in \Z_n^{\co{var}}\!
  }^{\text{\llap{Ex.} main cand.\ s.\ th.}}
\,\overbrace{%
  \!\!\!\!\!\!\!\!\!\!\!
  \bigwedge_{\substack{(Z_i)_i \text{ system}\\\text{of $\co{eq}$ disjoint}\\
  \text{zero-sums in }\Abar_{\star\star}}}\!\!
  }^{\text{\llap{for all} zero\rlap{-sum systems}}}  
\,\overbrace{%
  \bigvee_{(c_i)_i \in C_0}\,
  \bigwedge_{1 \le i \le \co{eq}}\,
  \sum_{\{j \mid a_j \in Z_i\}} \!\!\!\! k f_j = \pi(kc_{i} - z_{i})
  }^{\text{$B$ is zero-sum free}}
\]
We used big conjunctions and disjunctions $\bigwedge$ and $\bigvee$
as notation for some of the universal and existential quantifiers
to emphasise that their range is finite and independent of $n$.

Putting this formula into disjunctive normal form and moving the
existential quantifier inside the $\bigvee$, we get that 
there exists a pair $(\Abar, f)$ satisfying $(*)$ if and only if
at least one of a finite number of systems of linear equations
(with coefficients in $\Z$ not depending on $n$)
has a solution in $\Z_n$.

By Lemma~\ref{Lem:LinearSystems}, each system either contributes only
finitely many integers $n$ such that $(\Abar, f)$ satisfies $(*)$,
or the contributed set has the form
$\{n \mid (n,d) \in \calT\}$ for some integer $d$ and some set $\calT$
of divisors of $d$ containing $1$. The union of sets of this form again has
this form, so the first part of the theorem is proven.

Concerning the algorithm it is enough to find computable bounds for
the following: a bound $n_0$ such that the above formula holds for all
$n \ge n_0$; a bound $n_1$ such that if the system
of equations is solvable modulo $n$ only for finitely many $n$, then
these $n$ are at most $n_1$; a bound $d_0$ such that
if the system of equations is solvable for infinitely many $n$,
then $d \le d_0$.

Clearly, all bounds which appear in this proof are computable, so
we do get this result. In Section~\ref{subsect:bounds}, we will
even determine such bounds explicitly.
\end{enumerate}

Now let us fill in the three remaining steps.

\begin{enumerate}
\item
Let $\Abar \subset \Z_{k}^{\ell}$ be of cardinality $kn + \delta$,
and suppose $Z\subset \Abar$ is any zero-sum subset. We will construct
a large system $(Z_i)_i$ of disjoint zero-sums in $\Abar$ such that $Z$ can
be written as union of some of these zero-sums $Z_i$. This then implies the
first step: if $B((Z_i)_i)$ is zero-sum free, then in particular the
sum $\sum_{a\in Z}f(a)$ is not zero.

By Lemma~\ref{Lem:Dm} we
can find at least $\lfloor\frac{|Z| - c(k, \ell)}k\rfloor$
disjoint zero-sums in $Z$ and at least
$\lfloor\frac{|\Abar \setminus Z| - c(k, \ell)}k\rfloor$
disjoint zero-sums in $\Abar \setminus Z$. We may suppose that $Z$
is the union of the zero-sums we found inside.
Together, we get 
$\lfloor\frac{|Z| - c(k, \ell)}k\rfloor
+ \lfloor\frac{|\Abar \setminus Z| - c(k, \ell)}k\rfloor
\ge \lfloor\frac{|\Abar| - 2c(k, \ell)}k\rfloor - 1 =: m =: n -\co{defect}$
disjoint zero-sums in $\Abar$. Note that $\co{defect}$ does not depend on $n$.
\item
Now suppose $(\Abar,f)$ is a candidate satisfying (*).
We want to show that in systems of $m$ disjoint zero-sums of $\Abar$,
almost all sets have exactly $k$ elements.

Suppose first that $\Abar$ contains $N$ disjoint zero-sum sets which together
have only $kN-c$ elements (for some value $c$). Then in the remaining
$k(n-N) + \delta + c$ elements of $\Abar$, we can find
(by Lemma~\ref{Lem:Dm})
$\lfloor\frac{k(n-N) + \delta + c - c(k, \ell)}k\rfloor
= n - N + \lfloor\frac{\delta + c - c(k, \ell)}k\rfloor$
disjoint zero-sums. If $c \ge c(k, \ell) - \delta =: \co{less} + 1$
these are $n-N$ disjoint zero-sums, and together with the
other $N$ ones, we get $n$ disjoint zero-sums $Z_i$. But then the set of sums
$B((Z_i)_i) \subset \Z_{n}$ can not be zero-sum free, which is a
contradiction.

In particular, we just showed that there are at most $\co{less}$ disjoint
zero-sum subsets of $\Abar$ with cardinality less than $k$.

Now let $(Z_{i})_{i}$ be a system of $m$ disjoint zero-sum sets. 
To see that almost all of these sets have at most $k$ elements, just
note that there are not so many elements in $\Abar$ left over to make the
sets bigger. More precisely, suppose that $M$ of the sets $Z_{i}$ have
more than $k$ elements, i.e.\ at least $k+1$ elements each. The remaining
$m-M$ sets contain at least $k(m-M) - \co{less}$ elements, so altogether we
get the inequality $M(k+1) + k(m-M) - \co{less} \le |\Abar| =
kn + \delta$. This implies
$M  \le kn + \delta - km + \co{less} =
\delta+k\cdot\co{defect} + \co{less} =: \co{more}$.

Putting both together, we get that no system of $m$ disjoint zero-sums has more
than $\co{card} := \co{more} + \co{less}$ sets of cardinality different from
$k$.
\end{enumerate}

The third step requires some more work, so we decompose it into several substeps.
We suppose that $(\Abar, f)$ is a candidate satisfying (*). In the first
three substeps, we prove some properties of $(\Abar, f)$; in the last substep,
we use this to construct another candidate $(\Abar', f')$ which will be a main
candidate satsifying (*).

\begin{enumerate}
\renewcommand{\theenumi}{3.\arabic{enumi}}
\item
Claim: Suppose that $n$ is sufficiently large. Then
for any system $(Z_i)_i$ of $m$ disjoint zero-sums in $\Abar$,
almost all elements of the sum-set $B := B((Z_i)_i)$
are equal to one single element $b \in \Z_n$ which generates $\Z_{n}$.

This follows from Lemma~\ref{Lem:CD}. We need $|B| = n - \co{defect} \ge
\frac{2}{3}n$, i.e.\ $n \ge 3\co{defect}$. And we get an element $b$ with
multiplicity at least $2|B| - n + 1 = m - \co{defect} + 1=: m - \co{ws}$.
(ws = wrong sum.)

\item
Claim: If $n \gg 0$, then the prevalent value $b$ in $B((Z_i)_i)$
is the same for any system $(Z_i)_i$ of $m$ disjoint zero-sums of $\Abar$.

Suppose $(Z_{i})_{i}$ and $(Z'_{i})_{i}$ are two different systems of disjoint
zero-sums, and denote the prevalent values is $B((Z_{i})_{i})$ and $B((Z'_{i})_{i})$ by $b$ and $b'$
respectively. We choose $\co{ws}+1$ of the sets $Z_{i}$ which all have cardinality
at most $k$ and all have $\Z_{n}$-sum $b$. This is possible if
$m \ge \co{more} + 2\co{ws}+1$.
Without loss, our chosen sets are $Z_{1}, \dots, Z_{\co{ws}+1}$.

Now we do the same for $(Z'_{i})_{i}$, i.e.\ we choose
$Z'_{1}, \dots, Z'_{\co{ws}+1}$ to have at most $k$ elements each and to have 
$\Z_{n}$-sum-values $b'$. But in addition, we want that these sets
$Z'_{j}$ (for $j \le \co{ws}+1$) are disjoint from the sets 
$Z_{i}$ (for $i \le \co{ws}+1$). Each set $Z_{i}$ can intersect at most
$k$ of the sets $Z'_{j}$, so the additional condition forbids at most
$k\cdot (\co{ws}+1)$ of the $m$ sets $Z_{j}$. Therefore we can
find our desired sets if $m \ge \co{more} + 2\co{ws} + 1 + k\cdot (\co{ws}+1)$.

Now we use Lemma~\ref{Lem:Dm} to complete our chosen sets
$(Z_{i})_{i \le\co{ws}+1}$ and $(Z'_{i})_{i \le\co{ws}+1}$
to a system of $m$ disjoint zero-sum sets.
By (3.1), there is a prevalent value $b''$
for this system, which leaves out at most
$\co{ws}$ sets. This implies that both $b$ and $b'$ are equal to $b''$.

Without loss, we will now suppose that the prevalent $\Z_n$-value of
any $m$ disjoint zero-sums is $1$.

\item
Claim: There exists a constant $\co{var}$ such that for at most $\co{var}$
of the elements $a \in \Abar$, we have $f(a) \ne \frac{1}{k}$.
In fact we will choose $\co{var}$ such that even a slightly stronger statement
holds: for each $a \in \Z_k^\ell$,
let $r_a$ be number of copies of $a$ in $\Abar$ with $f(a) = \frac{1}{k}$.
Then $\sum_{a \in \Z_k^\ell}k\cdot\lfloor\frac{r_a}{k}\rfloor \ge |\Abar| -
\co{var}$.

Let us call a subset $Z \subset \Abar$ ``neat'' if it is of the form
$\{a^k\}$ for some $a \in \Z_k^\ell$.

We construct a system $(Z_i)_i$ of $m$ disjoint zero-sums with lots of neat sets in the following way:
for each element $a \in \Z_k^\ell$ which appears with multiplicity $\mu$ in
$\Abar$, we form $\lfloor\frac{\mu}{k}\rfloor$ disjoint sets of the form
$\{a^k\}$. If we get more than $m$ sets in this way, we choose $m$ of them. If
we get less than $m$ sets, then we use Lemma~\ref{Lem:Dm} on the remainder
of $\Abar$ to complete our system $(Z_i)_i$. Denote by $\kappa$ the number of neat sets
in  $(Z_i)_i$.

The minimal value of $\kappa$ is attained if the multiplicity in $\Abar$
of each $a \in \Z_k^\ell$ is congruent $k-1$ modulo $k$. So we get
$\kappa \ge \min\{m, \frac1k(|\Abar| - (k-1)\cdot k^\ell)\} =: m - \co{nn}$
(nn = not neat).

Among all systems of $m$ disjoint zero-sums in $\Abar$ which have $\kappa$
neat sets, now choose a system $(Z_i)_i$ where the number of neat sets $Z_i$ with
sum $\sum_{a\in Z_{i}}f(a)$ equal to $1$ is minimal. At most
$\co{ws}$ sets have not sum $1$ and at most $\co{nn}$ are not neat,
so even in this minimal choice we get at least $m - \co{nn} - \co{ws}$ neat sets
with sum 1. We fix this system $(Z_i)_i$ for the remainder of step (3.3).

Choose $a \in \Z_k^\ell$, and let
$\mathcal{N}_a$ be the union of all neat sets $Z_i$ of the form $\{a^k\}$ with
$\Z_n$-sum $1$.
We claim that if there are at least two such neat sets, then $f$ is constant on
$\mathcal{N}_a$; in particular this implies that the value of
$f$ on $\mathcal{N}_a$ is $\frac{1}{k}$. Suppose $f$ is not constant on
$\mathcal{N}_a$. Then there are two elements $a_1, a_2\in \mathcal{N}_a$ with
$f(a_1) \ne f(a_2)$ which belong to two different neat sets $Z_{i_1}$,
$Z_{i_2}$. Modify the system $(Z_i)_i$ by exchanging $a_1$ and $a_2$. Then 
$Z_{i_1}$ and $Z_{i_2}$ do not have sum 1 anymore, so the new system contradicts
the assumption that the old one had a minimal number of neat sets with sum 1.

Doing the above construction for all $a \in \Z_k^\ell$ yields the claim:
The union $\mathcal{N} := \bigcup_{a \in \Z_k^\ell}\mathcal{N}_a$ contains all
neat sets $Z_{i}$ with $\Z_{n}$-sum $1$, so it has cardinality at least $k(m - \co{nn} -
\co{ws})$. On the other hand, if $f$ is not constant equal to $\frac{1}{k}$
on a set $\mathcal{N}_a$, then $|\mathcal{N}_a| = k$, and this can happen for at most
$k^\ell - 1$ of these sets.
Thus $f$ is equal to $\frac{1}{k}$ on at least
$k(m - \co{nn} - \co{ws}) - k(k^\ell - 1) =: |\Abar| - \co{var}$ elements.
As these elements are contributed in groups of $k$, we also get the
slightly stronger statement mentioned at the beginning of this step.

\item
Claim: There is a main candidate $(\Abar', f')$ satisfying (*)
(still assuming that $(\Abar, f)$ is an arbitrary candidate satisfying (*)).

Recall that 
$(\Abar', f')$ is a main candidate if there is an element $a_0\in \Z_k^\ell$
such that $\Abar'$ contains at least $|\Abar| - \co{var}$ copies $a$
of $a_0$ which moreover satisfy $f(a) = \frac{1}{k}$.

We construct $(\Abar', f')$ out of $(\Abar, f)$ in the following way.
As before, for $a\in \Z_k^\ell$ let $r_a$ be number of copies
$a' \in \Abar$ of $a$ with $f(a') = \frac{1}{k}$.
Choose $a_0\in \Z_k^\ell$ such that $r_{a_0}$ is maximal;
in particular $r_{a_0} \ge \frac{|\Abar| - \co{var}}{k^\ell}$. Let 
$(\Abar', f')$ be equal to $(\Abar, f)$ with the following modification:
For each $a \in \Z_k^\ell$, replace $k\cdot\lfloor\frac{r_a}{k}\rfloor$
copies $a' \in \Abar$ of $a$ satisfying $f(a') = \frac{1}{k}$ by
the same number of
copies $a''$ of $a_0$, and set $f'(a'') = \frac{1}{k}$ on these copies.
Denote by $\phi$ the bijection from $\Abar$ to $\Abar'$
which describes these replacements.

Step (3.3) ensures that $(\Abar',f')$ is a main candidate;
it remains to show that it satisfies (*).
To this end, for any zero-sum $Z' \subset \Abar'$, we construct
a zero-sum $Z \subset \Abar$ which has the same $\Z_{n}$-sum as $Z'$.
As $(\Abar, f)$ satisfies (*), this $\Z_{n}$-sum is not equal to zero, so
$(\Abar', f')$ satisfies (*), too.

So suppose a zero-sum $Z' \subset \Abar'$ is given.
Consider the set $\mathcal{M} \subset \Abar'$
of copies $a'$ of $a_{0}$ with $f'(a') = \frac{1}{k}$, and for $a \in
\Z_{k}^{\ell}$ define the subset
$\mathcal{M}_{a} := \{a' \in \mathcal{M} \mid \phi^{-1}(a') \text{ is a copy of
}a\}$. As $|\mathcal{M}_{a}|$ is a multiple of $k$ for any $a \ne a_{0}$,
and assuming $|\mathcal{M}_{a_{0}}| = r_{a_{0}} \ge k-1$, in $Z'$ we may
replace elements of $\mathcal{M}$ by other elements of $\mathcal{M}$
such that $|\mathcal{M}_{a} \cap Z'|$ is a multiple of $k$ for any
$a \ne a_{0}$. (This changes neither the sum nor the $\Z_{n}$-sum of $Z'$.)
Now take $Z := \phi^{-1}(Z')$. As elements are moved by groups of $k$,
$Z$ has the same sum as $Z'$ (i.e. zero),
and as $f' \circ \phi = f$, it has the same $\Z_{n}$-sum.
\qedhere
\end{enumerate}

\end{proof}

\subsection{Computation of the bounds}
\label{subsect:bounds}

The proof of Theorem~\ref{thm:decidable} actually gives a little more
than just decidability. In fact, for each $k$, $\ell$ and $\delta$,
there is a computable constant $n_0$, such that
$D(\Z_k^{\ell}\oplus\Z_{n})\le \delta + kn$ holds true for
all integers $n$ co-prime to $k$ if and only if it holds true for all
integers $n\leq n_0$ which are co-prime to $k$. In this subsection we
compute an upper bound for $n_0$ (Proposition~\ref{prop:bound}). Unfortunately, $D(G)$ is computable
only for very small groups $G$, while the value for $n_0$ obtained in
this subsection is rather large. However, we still believe that the
algorithm given above can be performed for several small values of $k$
and $\ell$, in particular if one does some manual improvements using
the explicit knowledge of $k$ and $\ell$.

We now compute all bounds appearing in the proof of Theorem~\ref{thm:decidable}.

A bound for Lemma~\ref{Lem:Dm}: Denote by $D^k(\Z_k^\ell)$ the least integer $n$
such that every multi-set consisting of $n$ elements in $\Z_k^\ell$
contains a zero-sum of length $\leq k$. Then $c(k, \ell)\leq
D^k(\Z_k^\ell) - k$, since every multi-set containing $k(m-1)+D^k(\Z_k^\ell)$
elements contains a system of $m$ disjoint zero-sums each of length
$\leq k$. For $D^k(\Z_k^\ell)$ we have the trivial bound $k^{\ell+1}$,
but also the estimate $D^k(\Z_k^\ell)\leq (256\ell\log\ell)^\ell\cdot
k$ due to Alon and Dubiner \cite{AD}. For specific values of $k$ and
$\ell$, great improvements on both bounds are possible; it is probably
at this point that our estimates can be improved most easily. To avoid
some awkward expressions in the sequel, we shall express all constants
occurring in the proof of Theorem~\ref{thm:decidable} explicitly in
terms of $k$, $\ell$, $\delta$ and $c(k, \ell)$, and give an explicit estimate
using only the bound $c(k, \ell)\leq k^{\ell+1}$. (For the explicit
estimates, we use that we may suppose $k \ge 2$, $\ell \ge 3$, $\delta \ge 2$.)

Step (1):  $\co{defect} = 1 + \lceil \frac{2c(k, \ell) -
\delta}{k}\rceil \leq 3k^{\ell}$

Step (2): $\co{less} = c(k,\ell) - \delta - 1\leq k^{\ell+1} - \delta$

Step (2): $\co{more} = \delta+k\cdot\co{defect} +
\co{less}\leq 4k^{\ell+1}$

Step (2): $\co{card} = \co{more} + \co{less}\leq 5k^{\ell+1}$

Step (3.1): $\co{ws} = \co{defect} - 1\leq 3k^{\ell}$

Step (3.1) needs $n \ge 3\co{defect}$.
So $n\geq 9k^{\ell}$ suffices.

Step (3.2) needs $n \ge \co{defect} + \co{more} + 2\co{ws} + 1 +
k\cdot (\co{ws}+1)$. So $n\geq 12k^{\ell+1}$ suffices.

Step (3.3): $\co{nn} = \max\{0, (k-1)\cdot
k^{\ell-1}-\frac1k\delta -\co{defect}\}$. The proof of 
Theorem~\ref{thm:decidable} allows us to assume
$\co{defect} = 3k^{\ell}$, which yields $\co{nn} = 0$.
(However, using more careful estimates for $c(k, \ell)$ could yield
non-zero values for $\co{nn}$.)

Step (3.3): $\co{var} = \delta + k(\co{defect} + \co{nn} +
\co{ws} + k^\ell - 1) \leq 7k^{\ell+1} + \delta$

Step (3.4) needs $\frac{kn+\delta - \co{var}}{k^{\ell}} \ge k - 1$.
So $n \ge 8k^{\ell}$ suffices.

Step (5): $\co{eq} = \co{card} + \co{var} \leq 12k^{\ell+1} + \delta$

Step (7) needs $n \ge 4\co{defect}$. So $n\geq 12k^{\ell}$ suffices.

Step (7) also needs $n \ge 2(\co{defect} + \co{eq})$.
Here $n\geq 27k^{\ell+1}+2\delta$ is suffices. This is the largest
bound on $n$ of the proof.

\medskip

Concerning the systems of equations, we get:

Step (8): The coefficients of the equations are all equal to $k$.

Step (8): The absolute values of the right hand sides of the equations are bounded by
$\max(k(\co{defect} + \co{eq}), |\Abar_{\star\star} \setminus \Abar_{\star}|)
= k(\co{defect}+\co{eq}) \le 14k^{\ell+2} + k\delta$.

Step (9): The number of variables in each system of equations is $\co{var} \leq 7k^{\ell+1} + \delta$.

Step (9): The left hand side of any equation is of the form
$\sum_{j}kf_j$, where the sum runs over a subset of $\{1, \dots, \co{var}\}$;
thus we may suppose that each system of equation consists of at most
$2^{\co{var}}\leq 2^{7k^{\ell+1} + \delta}$ equations.

Hence, we can apply Lemma~\ref{Lem:LinearSystems} and
Corollary~\ref{Cor:LinearSystem} to obtain the following.

\begin{prop}\label{prop:bound}
There exists a constant $c$ such that the following holds true. Let
$k, \ell, \delta$ be integers, such that there
exists some $n$, co-prime to $k$, satisfying $D(\Z_k^{\ell}\oplus \Z_{n})>
\delta+kn$. Denote by $\calN$ the set of these $n$, and let $n_1$ be minimum of $\calN$.
Then we have
$n_1\leq  2^{2^{c(k^{\ell+1}+\delta)}}$. Moreover, if $\calN$ is infinite, then
we have $n_1\leq 6\ell(7k^{\ell+1}+\delta)\log k\delta$.
\end{prop}
\begin{proof}
Using the estimates above and Lemma~\ref{Lem:LinearSystems},
in the case that $\mathcal{N}$ is finite, we obtain the bound
\begin{eqnarray*}
n_1 & \leq &
(14k^{\ell+2} + k\delta)2^{p\big(2^{7k^{\ell+1} + \delta} \cdot
  (7k^{\ell+1} + \delta) \cdot \log k \big)}\\
 & \leq & 2^{2^{c(k^{\ell+1}+\delta)}},
\end{eqnarray*}
and our claim follows in this case. If $\mathcal{N}$ is infinite, we
additionally use Corollary~\ref{Cor:LinearSystem} to find that the systems
of linear equations are solvable for an
$n \in [z, 2z]$, provided that $z\ge \max(z_0, 21)$, where
\begin{eqnarray*}
z_0 & \leq & \textstyle{\frac{1}{\log2}}\co{var}\log(\co{var}k)\\
 & \leq & \textstyle{\frac{1}{\log2}}(7k^{\ell+1}+\delta)\log(7k^{\ell+2}+\delta k)\\
 & \leq & 3\ell(7k^{\ell+1}+\delta)\log k\delta,
\end{eqnarray*}
where we used the fact that we may suppose
$\ell\geq 3$, $\delta \ge 2$. Hence, $n_1\leq 2z_0$. To be sure to get an element
of $\calN$ in $[z, 2z]$, we moreover need $z \geq 27k^{\ell+1}+2\delta$, which
is less than the bound just computed. Thus there exists some $n \in \calN$ which
is at most two times our bound; this was our claim.
\end{proof}

Note that the smallest case of interest would be $k=4, \ell=3, \delta = 6$, that
is, checking $D(\Z_4^2\oplus\Z_{4n})=4n+6$ for all odd $n$ up to 3375 would
imply that this equation has only finitely many
counter-examples. Unfortunately, even the case $n=3$ has not yet been decided,
although it is within reach of modern computers.

\section{Proof of Theorem~\ref{thm:Davenport}}
In this section we prove that $B(n)$ implies $D(\Z_3\oplus\Z_{3n}^2)=6n+1$
if $n$ is co-prime to $6$.
We suggest that before reading the following lemmas,
the reader goes directly to the main proof and starts reading it to get
the main idea.

\subsection{Lemmas needed in the proof}

\begin{Lem}
\label{Lem:Length3}
Among 17 arbitrary elements in $\Z_3^3$ there is a zero-sum of length at most 3,
and among 9 \emph{distinct} elements there is a zero-sum of length at most 3.
Moreover, up to linear equivalence, there is precisely one set of 8
distinct elements without zero-sums of length at most 3,
which is given as $\{x, y, z, x+y,
x+y+z, x+2y+z, 2x+z, y+2z\}$.
\end{Lem}
\begin{proof}
The second part is \cite[Lemma~1 (ii)]{Montreal}, the first part is
folklore (and follows immediately from the second part).
\end{proof}

\begin{Lem}\label{Lem:completions}
Suppose that $n\ge 5$ is an integer having property $B$,
and $B$ is a subset of $\Z_n^2$ with either $2n-3$ or $2n-4$ points.
Then, with one exception, there always exists a group homomorphism
$F\colon \Z_n^2\to \Z_n$ such that:
\begin{enumerate}
\item
In the case $|B| = 2n-3$: For any $c$ with $B \cup \{c\}$ zero-sum free,
we have $F(c) = 1$.
\item
In the case $|B| = 2n-4$: For any $c_1,c_2$ with $B \cup \{c_1,c_2\}$ zero-sum free,
we have $F(c_i) \in \{0,1\}$, and at least one of $F(c_1)$ and $F(c_2)$
is equal to $1$.
\end{enumerate}
The exception is $B = \{b_{1}^{n-2}, b_{2}^{n-2}\}$, where
$b_{1}$ and $b_{2}$ generate $\Z_n^2$.
\end{Lem}

\begin{proof}
Every completion of $B$ to a zero-sum free set contains an element $b$
with multiplicity $n-2$ or $n-1$ such that all other elements of the
completion are contained in a Co-set of $\langle b\rangle$
which is a generator of $\Z_n^2/\langle b\rangle$.
We will call an element of $B$ \emph{important} if it could get such an
element after completion; i.e.\ an element $b\in B$ is important if its
multiplicity is at least $n-3$ in the first case or $n-4$ in the second case,
if its order is $n$ and if all other elements of $B$ are contained in
a Co-set of $\langle b\rangle$ which is a generator of $\Z_n^2/\langle b\rangle$.
$B$ contains at least one important element.
We will do case distinctions
between the different possibilities for the important elements of $B$.
But before we start, let us have a closer look at what can happen if $B$
contains two important elements, say $b_1$ and $b_2$.

First note that these two elements generate $\Z_n^2$,
as (by the importance of $b_1$)
$b_2$ lies in a Co-set of $\langle b_1\rangle$ generating
$\Z_n^2/\langle b_1\rangle$.
Now $b_2$ fixes the Co-set of $\langle b_1\rangle$ and vice versa, so
all elements of $B$ other than $b_1$ and $b_2$ lie in both
$b_2 + \langle b_1\rangle$ and $b_1 + \langle b_2\rangle$; we get
$B = \{b_1^{m_1}, b_2^{m_2}, (b_1+b_2)^{|B|-m_1-m_2}\}$. In particular,
$B$ contains no third important element.

First consider the case $|B| = 2n-3$. We distinguish the following
cases:
\begin{itemize}
\item
  $B$ contains only one important element $b$. Then the other elements
  of $B$ define a Co-set $L$ of $\langle b\rangle$, and all elements $c$
  completing $B$ either are equal to $b$ or lie in $L$.
  If $b$ has multiplicity $n-1$, then $c = b$ is impossible, so choose
  $F$ such that $F(L) = 1$.
  If $b$ has multiplicity $n-2$, then there are only two possibilities
  for $c$: $c = b$ and one other possibility on $L$ (such that the sum
  of $c$ and the elements of $B \cap L$ is equal to $b$). Choose $F$ to be 1 on these
  two possibilies.
  If $b$ has multiplicity $n-3$, then only $c = b$ is possible.
\end{itemize}
In the remaining cases, 
$B$ contains two important elements, so
$B = \{b_1^{m_1}, b_2^{m_2}, (b_1+b_2)^{m_3}\}$
for some $m_1,m_2,m_3$ satisfying and $m_1+m_2+m_3 = 2n-3$.
We may suppose $m_1 \ge m_2$.
\begin{itemize}
\item
  $m_1 = n-1$: All completions of $B$ lie in $b_2 + \langle b_1\rangle$.
\item
  $m_1=m_2=n-2$, $m_3=1$: There are two possible completions: $c=b_1$
  and $c=b_2$.
\item  
  $m_1=n-2$, $m_2=n-3$, $m_3=2$: There are two possible completions:
  $c=b_1$ and $c=b_2-b_1$.
\item
  $m_1=m_2=n-3$, $m_3=3$: There is no possible completion.
\end{itemize}

Now consider the case $|B| = 2n-4$. We distinguish the following cases:
\begin{itemize}
\item
  $B$ contains only one important element $b$. Then the other elements
  of $B$ define a Co-set $L$ of $\langle b\rangle$, 
  and for all completions $\{c_1,c_2\}$, both $c_i$ lie in $L \cup \{b\}$.
  If the multiplicity of $b$ in $B$ is $n-1$ or $n-2$, we can take $F$
  to be the function which is 1 on $L$ (and $0$ on $b$). Otherwise
  at least one of the $c_i$ is equal to $b$ and the other one either
  es equal to $b$, too, or it lies on $L$ and is determined by $B$.
  So a function $F$ exists.
\end{itemize}
Again, in the remaining cases $B = \{b_1^{m_1}, b_2^{m_2}, (b_1+b_2)^{m_3}\}$
with $m_1 \ge m_2$ and $m_1+m_2+m_3 = 2n-4$.
\begin{itemize}
\item
  $m_1 = m_2 = n-2$, $m_3 = 0$. This is the exception mentioned in the statement
  of the lemma.
\item
  $m_1=n-2$, $m_2\le n-3$: There are three types of completions:
  $c_1 = b_1$ and $c_2 \in b_2+\langle b_1\rangle$;
  $c_1 = c_2 = b_2$; both $c_i$ lie in $b_2+\langle b_1\rangle$
  with some condition on $c_1+c_2$. (Note that in the case $m_2 = n-3$,
  we have $m_3 = 1$ and $c_1=b_2$ implies $c_2 = b_1$.)
  So the function $F$ which maps
  $b_2+\langle b_1\rangle$ to $1$ does the job.
\item
  $m_1 = m_2 = n-3$, $m_3 = 2$: There are four possible completions:
  $\{b_1^2\}$, $\{b_2^2\}$, $\{b_1, b_2-b_1\}$ and $\{b_2, b_1-b_2\}$.
  Take $F$ to map $b_1$ and $b_2$ to $1$.
\item
  $m_1=n-3$, $n_2=n-4$, $m_3=3$: There are two possible completions:
  $\{b_1^2\}$ and $\{b_1, b_2-2b_1\}$. (Note that $\{b_2^2\}$ does not work.)
  Take $F$ to map $b_1$ and $b_2-2b_1$ to 1.
\item
  $m_1 = m_2 = n-4$, $m_3 = 4$: No completion is possible.
\end{itemize}
\end{proof}

We will need the following refined version of part 2 of
Lemma~\ref{Lem:completions}:

\begin{Lem}\label{Lem:compGraph}
Suppose that $n\ge 5$ is an odd integer having property $B$.
Suppose further that $B$ is a subset of $\Z_n^2$ with $2n-4$ points.
Let $C$ be the set of two-element-sets $\{c_1,c_2\} \subset \Z_n^2$ such that 
$B \cup \{c_1,c_2\}$ is zero-sum free.
Then, up to an automorphism of $\Z_n^2$, $C$ is a subset of
one of the following sets:
\begin{enumerate}
\item
$C_1 = \big\{\{(x_1,1), (x_2,1)\}\mid x_1, x_2\in \Z_n\big\}$.
\item
$C_2 = C_2' \cup C_2''$ with 
$C_2' = \big\{\{(1,0), (x,1)\}, \{(x,1), (1-x,1)\}\mid x\in \Z_n\big\}$
and
$C_2'' = \big\{\{(0,1), (1,y)\}, \{(1,y), (1,1-y)\}\mid y\in \Z_n\big\}$.
\item
$C_3 = C_3' \cup C_3''$ with 
$C_3' = \big\{\{(1,0)^2\}, \{(1,0), (-1,1)\}\big\}$ and
$C_3'' = \big\{\{(0,1)^2\}, \{(0,1), (1,-1)\}\big\}$.
\end{enumerate}
\end{Lem}
\begin{proof}
As in the proof of Lemma~\ref{Lem:completions}, we consider the different
possibilities for the important elements. If $B$ contains only one important element,
we can suppose that it is $(1,0)$ and that the other elements of $B$ have
$y$-coordinate one; we denote the multiplicity of $(1,0)$ by $m_1$.
If there are two important elements, we suppose that
$B = \{(1,0)^{m_1}, (0,1)^{m_2}, (1,1)^{m_3}\}$ with $m_1 \ge m_2$.
\begin{itemize}
\item
One important element, $m_1 = n - 1$: $C = C_1$.
\item
One important element, $m_1 = n - 2$: apply an automorphism of $\Z_n^2$
fixing $(1,0)$ and mapping the sum of those $n-2$ elements of $B$ with
$y$-coordinate one to $(0,-2)$. Then $C = C_2' \subset C_2$.
\item
One important element, $m_1 = n - 3$: apply an automorphism
fixing $(1,0)$ and mapping the sum of those $n-1$ elements of $B$ with
$y$-coordinate one to $(2,-1)$. Then $C = C_3' \subset C_3$.
\item
One important element, $m_1 = n - 4$: $C =  \big\{\{(1,0)^2\}\big\}\subset C_3$.
\item
Two important elements, $m_1=m_2=n-2, m_3=0$: $C = C_2$.
\item
Two important elements, $m_1=n-2, m_2=n-3, m_3=1$: apply an automorphism
fixing $(1,0)$ and mapping $(0,1)$ to $(\frac12,1)$. Then $C = C_2'\subset C_2$.
\item
Two important elements, $m_1=n-2, m_2=n-4, m_3=2$: apply an automorphism
fixing $(1,0)$ and mapping $(0,1)$ to $(1,1)$. Then $C = C_2'\subset C_2$.
\item
Two important elements, $m_1=m_2=n-3, m_3=2$: $C = C_3$.
\item
Two important elements, $m_1=n-3, m_2=n-4, m_3=3$: apply an automorphism
fixing $(1,0)$ and mapping $(0,1)$ to $(1,1)$. Then $C = C_3'\subset C_3$.
\item
Two important elements, $m_1=m_2=n-4, m_3=4$: $C = \emptyset$.
\end{itemize}
\end{proof}

In addition, we will need the two following lemmas:

\begin{Lem}
\label{Lem:NoFunc1}
Suppose $n$ is an integer co-prime to 6 and
$\Abar\subseteq \Z_3^3$ has $10$ elements.
Suppose further that 
$\Abar$ has no zero-sum of length $\leq 3$ and $\Abar$ has no two disjoint
zero-sums.
Then there is no multi-function $g\colon \Abar\to\Z_n$
(i.e.\ function which may take different values on different copies
of an element $a \in \Abar$) such that
for every zero-sum $Z\subseteq \Abar$ we have
$\sum_{z\in Z}g(z)=1$.
\end{Lem}
\begin{proof}
If we would require $g$ to be a real (i.e. single-valued) function,
then this would be \cite[Theorem~1]{Montreal}. So the only thing we have to
check is that the existence of a multi-function $g$ implies the existence
of a real function $g'$ with the same properties.

Define $g'$ by taking for $g(a)$ the mean value of
the values of $g(a)$. Note first that the maximal multiplicity of points in $\Abar$
is 2 (as $\Abar$ does not contain a zero-sum of length 3), so $g$
can have at most two values at any point. In particular
the mean value makes sense (because $2 \nmid n$).

Now consider any point $a \in \Abar$ where $g$ has two values.
The modification does not change $\sum_{z\in Z}g(z)$ if $Z$ does not contain
$a$ or if $Z$ contains both copies of $a$. However, no zero-sum $Z$ can contain
only one copy of $\Abar$, for otherwise, we would get two different values
for $\sum_{z\in Z}g(z)$, which contradicts $\sum_{z\in Z}g(z)=1$.
\end{proof}

\begin{Lem}
\label{Lem:NoFunc2}
Suppose $n$ is an integer co-prime to 6,
$\Abar\subseteq \Z_3^3$ has $13$ elements, and $f\colon \Abar\to\Z_n^2$
is a multi-function.
Suppose further that 
$\Abar$ has no zero-sum of length $\leq 3$ and $\Abar$ has no three disjoint
zero-sums.
Let $C$ be the set of two-element-sets
$\{\sum_{z\in Z_1}f(z), \sum_{z\in Z_2}f(z)\}$, where $Z_1$ and $Z_2$
are two disjoint zero-sums in $\Abar$. Then $C$ is not a subset of
any of the three sets $C_1$, $C_2$ or $C_3$ of Lemma~\ref{Lem:compGraph}.
\end{Lem}

\begin{proof}
This has been verified by our computer. For details on how this
has been done see Section~\ref{sect:Algo2}.

Note that concerning $C_1$, this is just an unnecessarily complicated way of
saying that there is no function $g \colon \Abar \to \Z_n$ which maps
any zero-sum of $\Abar$ which is disjoint to another zero-sum
to one.
\end{proof}

\subsection{The proof itself}

We are now in a position to prove Theorem~\ref{thm:Davenport}.

\begin{proof}[Proof of Theorem~\ref{thm:Davenport}]
Suppose $n$ is co-prime to $6$, $B(n)$ holds true,
$G=\Z_3\oplus\Z_{3n}^2$, and $A\subseteq G$ is a multi-set of
$M(G)=6n+1$ elements. Suppose $A$ contains no zero-sum. We have to get
to a contradiction.

Let $\Abar$ be the projection of $A$ onto $\Z_3^3$, and let $f\colon \Abar
\to \Z_n^2$ be the multi-function such that $(a, f(a))$ is the preimage
of $a \in \Z_3^3$ in $A$ under the projection.

We remove zero-sums of length $\leq 3$ from $\Abar$ as long as possible,
ending in a set $\Abar^*$ with less than 17 points (by Lemma~\ref{Lem:Length3}).
Denote by $B$ the multi-set
in $\Z_n^2$ corresponding to the removed zero-sums:
for each removed zero-sum $Z \subset \Abar$, put the element
$\sum_{z\in Z} f(z)$ into $B$. As $A$ is zero-sum free,
so is $B$. The strategy in the remainder of the proof is to
consider zero-sums $Z \in \Abar^*$ and their corresponding elements
$c = \sum_{z\in Z}f(z)$ in $\Z_n^2$. If we find such a $c$ such that
$B \cup \{c\}$ does contain a zero-sum, we have our desired contradiction.
When using this strategy, we may assume that while passing from $\Abar$ to
$\Abar^*$ we never removed zero-sums of length $< 3$; otherwise $\Abar^*$
only gets bigger and the proof gets easier.

Hence $|\Abar^*|$ has the form $3i+1$ and $|B| = 2n-i$.
As $B$ has no zero-sum, we have $|B| \le 2n-2$, so $i \ge 2$ and $|\Abar^*| \geq 7$.
If $|\Abar^*| = 7$, then $\Abar^*$ itself still contains
a zero-sum, so this is not possible either. Therefore $\Abar^*$ consists of
10, 13 or 16 points.

Suppose first that we end with $|\Abar^*| = 16$. Then we have 16 points without
a zero-sum of length $\le 3$. As 9 distinct points would contain such a zero-sum
(by Lemma~\ref{Lem:Length3}) there are precisely 8 points taken
twice. Since the only configuration of 8 distinct points without a
zero-sum of length 3 is the one given in Lemma~\ref{Lem:Length3}, we
find that $\Abar^*$ equals this set with each point taken twice. But this set
contains four disjoint zero-sums:
$\{x, y, (x+y)^{2}\}$,
$\{x, z^{2}, 2x+z\}$,
$\{y, x+y+z, (x+2y+z)^{2}\}$ and
$\{x+y+z, 2x+z, (y+2z)^{2}\}$. So we can enlarge $B$ to a set with $2n-1$
elements, which is a contradiction.

Next, suppose that $|\Abar^*|=10$. Then $B$ consists of $2n-3$
points in $\Z_n^2$, and each zero-sum $Z$ in $\Abar^*$ yields an element
$c = \sum_{z\in Z}f(z)$ of $\Z_n^2$ such that $B \cup \{c\}$ is zero-sum free.
Since $n$ satisfies property $B$ (and is $\ge 5$),
we can apply Lemma~\ref{Lem:completions} and obtain a linear function
$F\colon\Z_n^2\to\Z_n$ such that for every $c$ as above
$F(c) = 1$. But now $g := F \circ f$ is a contradiction to 
Lemma~\ref{Lem:NoFunc1}.

Finally, consider the case $|\Abar^*|=13$. Then $B$ consists of $2n-4$
points in $\Z_n^2$.
We check that $\Abar^*$ and $f$ contradict Lemma~\ref{Lem:NoFunc2}.
It is clear that $\Abar^*$ does not contain a zero-sum of length $\le3$
and that $\Abar^*$ does not contain three disjoint zero-sums.

Denote by $C$ the set of two-element-sets
$\{\sum_{z\in Z_1}f(z), \sum_{z\in Z_2}f(z)\}$, where $Z_1$ and $Z_2$
are two disjoint zero-sums in $\Abar^*$.
Each $\{c_1,c_2\} \in C$ completes $B$ to a zero-sum free subset of $\Z_n^2$,
so by Lemma~\ref{Lem:compGraph}, $C$ is a subset of
one of the three sets $C_i$ mentioned in that lemma. This is exactly what
we need to get a contradiction to Lemma~\ref{Lem:NoFunc2}.
\end{proof}

\section{Computer proof of Lemma~\ref{Lem:NoFunc2}}
\label{sect:Algo2}

Recall the statement of the lemma : we are given an integer
$n$ co-prime to 6, a set
$\Abar\subseteq \Z_3^3$ consisting of $13$ elements, and a
multi-function $f\colon \Abar\to\Z_n^2$.
We suppose that $\Abar$ has no zero-sum of length $\leq 3$ and no three
disjoint zero-sums.
We let $C$ be the set of two-element-sets
$\{\sum_{z\in Z_1}f(z), \sum_{z\in Z_2}f(z)\}$, where $Z_1$ and $Z_2$
are two disjoint zero-sums in $\Abar$.
The statement is that $C$ is not a subset of
any of the three sets $C_1$, $C_2$ or $C_3$ of Lemma~\ref{Lem:compGraph}:
\[
\begin{aligned}
C_1 =& \big\{\{(x_1,1), (x_2,1)\}\mid x_1, x_2\in \Z_n\big\}\\
C_2 =& \big\{\{(1,0), (x,1)\}, \{(x,1), (1-x,1)\}\mid x\in \Z_n\big\} \,\cup\\
&\big\{\{(0,1), (1,y)\}, \{(1,y), (1,1-y)\}\mid y\in \Z_n\big\}
\\
C_3 =& \big\{\{(1,0)^2\}, \{(1,0), (-1,1)\},\,\, \{(0,1)^2\}, \{(0,1), (1,-1)\}\big\}
\end{aligned}
\]

The program is divided into two parts. First find all possible
multi-sets $\Abar$ (up to automorphism of $\Z_3^2$),
regardless of the function $f$, and then, for each
fixed set $\Abar$ and each $i \in \{1,2,3\}$,
find all possible functions $f\colon \Abar \to \Z_n^2$
such that $C \subset C_{i}$. If no such $f$ is found,
then the lemma is proven.

\subsection{Finding all multi-sets $\Abar$}

The program recursively tries every possibility for $\Abar$ by starting
with an empty set and successively adding elements.
After adding an element, it checks right away if $\Abar$
still fulfils the above conditions before adding more elements.

To save some time, symmetry is exploited a bit. For example,
if $\Abar$ contains exactly two elements of multiplicity 2,
then we can suppose that $\Abar$ contains 
$(1,0,0)$ and $(0,1,0)$ with multiplicity 2 and $(0,0,1)$ with
multiplicity 1.

As we do not exploit symmetry completely (this would be too complicated),
the program finds a lot of solutions which are the same up to automorphism,
so we need an algorithm to check whether there is an automorphism
turning one multi-set into another one. It turns out that
all solutions $\Abar$ do contain a basis of $\Z_3^2$ of elements of order two,
so it is enough to try those automorphisms which map this basis
of one of the sets to elements of order two of the other set.

The program finds the following 15 multi-sets. The three $3\times3$-grids
represent the three planes of the cube $\Z_3^3$; the element $(0,0,0)$ is the lower
left corner of the left-most plane. The numbers in the grids indicate the multiplicity
of that element; empty squares mean that the element is not contained in the
set.

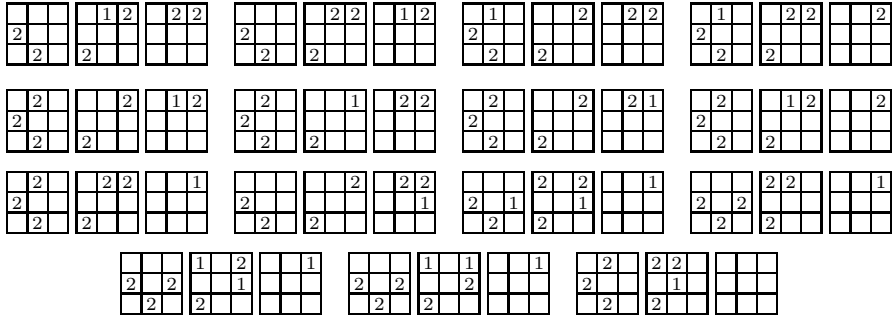
\[
\boix{}2{}2{}{}{}{}{}\;\boix2{}{}{}{}{}{}12\;\boix{}{}{}{}{}{}{}22
\quad
\boix{}2{}2{}{}{}{}{}\;\boix2{}{}{}{}{}{}22\;\boix{}{}{}{}{}{}{}12
\quad
\boix{}2{}2{}{}{}1{}\;\boix2{}{}{}{}{}{}{}2\;\boix{}{}{}{}{}{}{}22
\quad
\boix{}2{}2{}{}{}1{}\;\boix2{}{}{}{}{}{}22\;\boix{}{}{}{}{}{}{}{}2
\]\[
\boix{}2{}2{}{}{}2{}\;\boix2{}{}{}{}{}{}{}2\;\boix{}{}{}{}{}{}{}12
\quad
\boix{}2{}2{}{}{}2{}\;\boix2{}{}{}{}{}{}{}1\;\boix{}{}{}{}{}{}{}22
\quad
\boix{}2{}2{}{}{}2{}\;\boix2{}{}{}{}{}{}{}2\;\boix{}{}{}{}{}{}{}21
\quad
\boix{}2{}2{}{}{}2{}\;\boix2{}{}{}{}{}{}12\;\boix{}{}{}{}{}{}{}{}2
\]\[
\boix{}2{}2{}{}{}2{}\;\boix2{}{}{}{}{}{}22\;\boix{}{}{}{}{}{}{}{}1
\quad
\boix{}2{}2{}{}{}{}{}\;\boix2{}{}{}{}{}{}{}2\;\boix{}{}{}{}{}1{}22
\quad
\boix{}2{}2{}1{}{}{}\;\boix2{}{}{}{}12{}2\;\boix{}{}{}{}{}{}{}{}1
\quad
\boix{}2{}2{}2{}{}{}\;\boix2{}{}{}{}{}22{}\;\boix{}{}{}{}{}{}{}{}1
\]\[
\boix{}2{}2{}2{}{}{}\;\boix2{}{}{}{}11{}2\;\boix{}{}{}{}{}{}{}{}1
\quad
\boix{}2{}2{}2{}{}{}\;\boix2{}{}{}{}21{}1\;\boix{}{}{}{}{}{}{}{}1
\quad
\boix{}2{}2{}{}{}2{}\;\boix2{}{}{}1{}22{}\;\boix{}{}{}{}{}{}{}{}{}
\]
\subsection{Finding all functions $f\colon \Abar \to \Z_n^2$}

Now fix a set $\Abar$ as above and fix $C := C_1$, $C := C_2$
or $C := C_3$.
We have to check
that there is no function $f\colon \Abar \to \Z_n^2$ such that
for any pair of disjoint zero-sums $Z_1$ and $Z_2$ in $\Abar$,
the pair
$\{\sum_{z\in Z_1}f(z), \sum_{z\in Z_2}f(z)\}$ is contained in $C$.

This can be reformulated as follows. From $\Abar$, we define the following
graph $G = (V, E)$: the vertices $V$ are the zero-sums $Z \subset \Abar$ such that there
does exist a second zero-sum $Z' \subset \Abar$ which is disjoint to $Z$,
and the edges $E$ are the pairs $Z_1, Z_2 \in V$ which are disjoint.
The set $C$ defines another graph $G' = (V', E')$: $V'$ consists of all
elements which appear in some pair in $C$, and $E' = C$, i.e.\ the edges
are just the pairs contained in $C$.
Any function $f\colon \Abar \to \Z_n^2$ satisfying the above condition
defines a graph homomorphism $\phi\colon G \to G'$,
and a graph homomorphism $\phi\colon G \to G'$ yields a function $f$
if and only if the following system of linear equations $L_\phi$ has a solution
in $\Z_n$: we have two variables $x_i$ and $y_i$ ($i \in \{1, \dots, 13\}$) for
the two coordinates of each $f(a_i), a_i\in \Abar$, and for each
vertex zero-sum $Z = \{a_{i_1},\dots, a_{i_k}\} \in V$ we have the two equations
given by $\sum_{j=1}^{k} a_{i_j} = \phi(Z)$.

The idea of the algorithm is to try every graph homomorphism $\phi$ and
to check that the corresponding system of linear equations $L_\phi$
has no solution for any $n$ co-prime to $6$. But before we can do that,
we have to replace $G'$ by a simpler graph $G''$.

To simplify $G'$, we merge some of the points which differ only in one
coordinate. Thus a graph homomorphism $\phi\colon G \to G''$ will give
less equations in $L_\phi$. We do not ensure that these equations
are enough to prove the existence of $f$; we only need that if the
equations have no solution, then no $f$ exists.

In the case of $C_1$, all this graph homomorphism  is overkill
(as already noted directly after Lemma~\ref{Lem:compGraph}), but
let us formulate it anyway so that we can treat all three cases
similarly.

\begin{itemize}
\item
Case $C_3$: No simplification necessary; $G'' = G'$.
\item
Case $C_1$: Merge all points of $G'$ to one single point in $G''$
with a loop edge. Each zero-sum $Z \in V$ mapped to that point
(i.e.\ all $Z \in V$) yields one equation in
$L_\phi$ saying that the sum of the $y$-coordinates is equal to one.
\item
Case $C_2$: Merge all points $(1, y)$ for $y \ge 2$ into one point
and all points $(x, 1)$ for $x \ge 2$ into one point. So $G''$ looks like this:
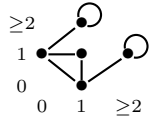
\[
\begin{array}{c@{}c@{\quad}c@{\quad}c}
\scriptstyle{\ge 2}&&\bnode{y2}\\
\scriptstyle{1}&\bnode{x0}&\bnode{xy}&\bnode{x2}\\
\scriptstyle{0}&&\bnode{y0}\\[-1ex]
&\scriptstyle{0}&\scriptstyle{1}&\scriptstyle{\ge 2}
\end{array}
\ncline{x0}{y0}
\ncline{x0}{xy}
\ncline{x0}{y2}
\ncline{y0}{xy}
\ncline{y0}{x2}
\nccircle[angle=-45]{y2}{1ex}
\nccircle[angle=-45]{x2}{1ex}
\]
Zero-sums which get mapped to $(1, 0)$, $(0, 1)$ or $(1, 1)$ still yield
two equations in $L_\phi$. Zero-sums which get mapped to $(1, \ge2)$
or $(\ge2, 1)$ yield only one equation saying that the sum
of the $x$-coordinates resp. $y$-coordinates is equal to 1.
In addition, we get equations for each edge which is mapped
to the loop at $(1, \ge2)$ (and, analogously, at $(\ge2, 1)$):
if $(1, y_1)$ and $(1, y_2)$ were connected in $G'$, then
$y_1+y_2 = 1$. So if $Z_1, Z_2 \in V$ are connected and are both mapped to
$(1, \ge2)$, then
the sum of the $y$-coordinates of all points in $Z_1 \cup Z_2$
is equal to 1.
\end{itemize}

Now our graph $G''$ is of reasonable size and it does make sense to
try every possible homomorphism $\phi\colon G \to G''$. This is done
by recursively fixing images $\phi(Z)$ for zero-sums $Z \in V$.
After an image is fixed, the algorithm first checks whether the equations
we already have do already yield a contradiction before going on.

The only thing left to describe is how to check whether a system of
linear equations has no solution in $\Z_n$ for any $n$ co-prime to $6$.
This could be done using the Smith normal form as in the proof of
Lemma~\ref{Lem:LinearSystems}, but this would probably be too slow.
Instead, we use the following method, which proves in sufficiently
many cases that no solution exists. (Note that we do not need an if-and-only-if
algorithm.)

We apply Gaussian elimination over $\Z$ to our system of equations and
then consider only 
the equations of the form ``$a = 0$'' for $a \ne 0$ which we get. Each such equation is
interpreted as a condition on $n$, namely ``$n$ divides $a$''.
If, taking all these equations together,
we get that $n$ has only prime factors $2$ and $3$, then we have a
contradiction.

The algorithm takes about one second in the case $C_1$, 70 minutes in the
case $C_2$, and 5 minutes in the case $C_3$ (for all 15 sets $\Abar$ together).

One more practical remark:
When recursively trying all possible maps $\phi\colon G \to G''$, we use
a slightly intelligent method to choose which $\phi(Z)$ to fix next:
if there is a $Z\in V$ for which there is only one possible image left, we take
that one; otherwise, we take a $Z\in V$ with maximal degree.

\begin{tabular}{lll}
Gautami Bhowmik, &  Immanuel Halupczok,& Jan-Christoph Schlage-Puchta,\\
Universit\'e de Lille 1,& DMA de l'ENS, & Albert-Ludwigs-Universit\"at,\\
Laboratoire Paul Painlev\'e,& UMR 8553 du CNRS, & Mathematisches Institut,\\
U.M.R. CNRS 8524,& 45, rue d'Ulm,  & Eckerstr. 1 \\
  59655 Villeneuve d'Ascq Cedex,& 75230 Paris Cedex 05, & 79104 Freiburg,\\
  France  & France& Germany \\
bhowmik@math.univ-lille1.fr & math@karimmi.de & jcp@math.uni-freiburg.de 
\end{tabular}


\begin{thebibliography}{99}
\bibitem{AD} N. Alon, M. Dubiner, A lattice point problem and additive
  number theory, {\em Combinatorica} {\bf 15} (1995), 301--309. 
\bibitem{Baa} P.C. Baayen, Eeen combinatorisch probleem voor eindige Abelse
  groepen, Colloq. Discrete Wiskunde caput {\bf 3} Math Centre, Amsterdam 
(1968).
\bibitem{Montreal} G. Bhowmik, J.-C. Schlage-Puchta, Davenport's
    constant for Groups of the Form $\Z_3\oplus\Z_3\oplus \Z_{3d}$,
    CRM Proceedings and Lecture Notes {\bf 43} (2007), 307--326.
\bibitem{Bmult} G. Bhowmik, I. Halupczok, J.-C. Schlage-Puchta, The
  structure of maximal zero-sum free sequences, preprint
\bibitem{BEN} J. D. Bovey, P. Erd\H os, I. Niven, Conditions for a
  zero sum modulo $n$, {\em Canad. Math. Bull.} {\bf 18} (1975), 27--29. 
\bibitem{Br} J. Br\"udern, H. Godinho, On Artin's conjecture
  I. Systems of diagonal forms. {\em Bull. London Math. Soc.} {\bf 31}
  (1999), 305--313. 
\bibitem{DOQ} C. Delorme, O. Ordaz, D. Quiroz, Some remarks on
  Davenport constant, {\em Discrete Math.} {\bf 237} (2001), 119--128.
\bibitem{GGComb} W. Gao, A. Geroldinger, On the structure of zerofree
  sequences, {\em Combinatorica} {\bf 18} (1998), 519--527.  
\bibitem{GaoB} W. Gao, A. Geroldinger, On zero-sum sequences in
  $\Z/n\Z\oplus\Z/n\Z$, {\em Integers} {\bf 3}  (2003), A8.
\bibitem{Gaop} W. Gao, A. Geroldinger, Zero-sum problems and coverings
  by proper cosets, {\em European J. Combin.} {\bf 24} (2003),
  531--549.
\bibitem{GGS} W. Gao, A. Geroldinger, W. A. Schmid, Inverse zero-sum
  problems, {\em Acta Arith.}, to appear.
\bibitem{GHK} A. Geroldinger, F. Halter-Koch, {\em Non-unique
    factorizations. Algebraic, combinatorial and analytic theory},
  Pure and Applied Mathematics (Boca Raton), 278. Chapman \& Hall/CRC,
  Boca Raton, FL, 2006.
\bibitem{KB} R. Kannan, A. Bachem,  Polynomial algorithms for computing the
  Smith and Hermite normal forms of an integer matrix.  {\em SIAM J. Comput.}
  {\bf 8}  (1979), no. 4, 499--507. 
\end{thebibliography}
\end{document}